[1994/12/01]

\documentclass[12pt]{article}

\title{Ideal Independence, Free Sequences, and the Ultrafilter Number}
\author{Kevin Selker	}
\date{December 14, 2012}

\usepackage{amsmath, amsthm, amssymb, amsfonts, euscript, amsxtra, mathrsfs, fullpage, enumerate  }

\chardef\bslash=`\\ 

\newtheorem{thm}{Theorem}[section]
\newtheorem{cor}[thm]{Corollary}
\newtheorem{lem}[thm]{Lemma}
\newtheorem{prop}[thm]{Proposition}

\theoremstyle{definition}
\newtheorem{defn}{Definition}[section]

\theoremstyle{remark}

\newcommand{\lemref}[1]{Lemma~\ref{#1}}

\newcommand{\X}{\mathcal{X}}
\newcommand{\se}{\subseteq}

\newcommand{\sm}{\setminus}

\newcommand{\uu}{\mathfrak{u}}

\newcommand{\smm}{\rm{s}_{\rm{mm}}}
\newcommand{\ff}{\mathfrak{f}}

\newcommand{\D}{\mathcal{D}}
\newcommand{\E}{\mathcal{E}}

\newcommand{\0}{\emptyset}

\newcommand{\w}{\omega}

\newcommand{\sdiff}{\bigtriangleup}

\newcommand{\ang}[1]{\left\langle #1 \right\rangle}
\newcommand{\paren}[1]{\left( #1 \right)}

\newcommand{\eval}[2][\right]{\relax
  \ifx#1\right\relax \left.\fi#2#1\rvert}

\let\abs=\envert

\providecommand{\fsub}[1]{[\nobreak#1\nobreak ]^{<\omega}}

\providecommand{\set}[1]{{\left\{#1\right\}}}

\newcommand{\pset}[1]{\mathscr{P} {(#1)}}

\renewcommand{\int}{\rm{int}}

\newcommand{\be}{\begin{enumerate}}
\newcommand{\ee}{\end{enumerate}}
\newcommand{\bi}{\begin{itemize}}
\newcommand{\ei}{\end{itemize}}

\begin{document} 
\maketitle

\begin{abstract}We make use of a forcing technique for extending Boolean algebras.  The same type of forcing was employed in \cite{bk}, \cite{kos}, and elsewhere.  Using and modifying a lemma of Koszmider, and using CH, we obtain an atomless BA, $A$ such that $\frak{f}(A) = \text{s}_{mm}(A) <\frak{u}(A),$ answering questions raised by \cite{monki} and \cite{monkf}. 
\end{abstract}


\emph{Classification:} 06E05, 54A25\\

This paper is concerned with some ``small'' cardinal functions defined on Boolean algebras.  To describe the results we need the following definition.  For notation concerning Boolean algebras we follow \cite{handbookba}

\setcounter{section}{1}

\begin{defn}\be \item A subset $Y$ of a BA is  ideal-independent if $\forall y\in Y , y\not\in\ang{Y\sm\set{y}}^{\text{id}}$.

\item We define $\smm(A)$ to be the minimal size of an ideal-independent family of $A$ that is maximal with respect to inclusion. 

\item A free sequence  in a BA is a sequence $X = \set{ x_\alpha: \alpha<\gamma}$ such that whenever $F$ and $G$ are finite subsets of $\gamma$ such that $\forall i\in F\forall j\in G[ i<j]$, then   $$ \paren{\prod_{\alpha\in F}{x_\alpha} }\cdot \paren{\prod_{\beta\in G}-x_\beta} \neq 0.$$  Where empty products equal 1 by definition. 

\item We define $\frak{f}(A)$ to be the minimal size of a free sequence in $A$ that is maximal with respect to end-extension. 

\item We define $\frak{u}(A)$ to be the minimal size of a nonprincipal ultrafilter generating set of $A$. 

\item If $A$ is a Boolean algebra and $u$ is a nonprincipal ultrafilter on $A$, let  $P(A, u)$ be the partial order consisting of pairs $(p_0, p_1)$ where $p_0, p_1\in A\sm u$, and $p_0\cap p_1 = \0$, ordered by $(p_0, p_1)\leq (q_0, q_1)$ (``$(p_0, p_1)$ is stronger than  $(q_0, q_1)$'') iff $q_i\se p_i$ for $i = 0, 1$. 

\ee
\end{defn}

The main result of the paper is that under CH there is an atomless BA $A$ such that $\w = \ff(A) = \smm(A)<\uu(A)  = \w_1$. Theorem 2.10 in \cite{monki} asserts the existence of an atomless BA with $\smm(A)<\uu(A),$ but the proof is faulty.  The existence of an atomless BA $A$ with $\ff(A) <\uu(A)$ is a problem raised in \cite{monkf}.

From now on, unless specified otherwise, $A$ will denote a subalgebra of $\pset{\kappa}$ for $\kappa$ an infinite cardinal and $u$ will denote a nonprincipal ultrafilter on $A$. 

\begin{lem}{\label{ina}}Suppose that $G$ intersects every dense subset of $P(A, u)$.  Let $g = \bigcup_{(p_0, p_1)\in G} p_0$.  Let $e,f\in A$.  Suppose that for some $p\in G$ we have $e\sdiff f \se p_0\cup p_1$. Then the set $b: = (g\cap e) \cup (f\sm g)$ is a member of $A$. 
\end{lem}
\begin{proof}  First we claim $g\cap (e\sm f) = p_0\cap (e\sm f)$.   In fact, $\supseteq $ is clear. Now suppose that $q\in G$; we want to show that $q_0 \cap (e\sm f)\se p_0\cap(e\sm f)$.  Choose $r\in G$ such that $r\leq q, p$. Note that $r_0\cap p_1\se r_0\cap r_1 = 0.$ Hence

$$ q_0\cap (e\sm f ) \se r_0 \cap  (p_0\cup p_1 ) \cap (e\sm f ) \se p_0 \cap (e\sm f ).  $$ Second we claim  $(-g) \cap (f\sm e) = (p_1\cap (f\sm e))$.

For $(\supseteq)$ suppose that $q \in G$; we want to
show that $q_0 \cap p_1 \cap(f \sm e) = 0.$ Choose $r \in G $ such that $r \leq p, q.$ Then $r_0 \cap p_1 \se r_0 \cap r_1 = 0$,
so $q_0 \cap p_1 \cap (f \sm e) \se r_0 \cap p_1 = 0.$

For $(\se)$, we have $f \sm e \se p_0 \cup p_1 ,$ so $-p_1 \se p_0 \cup -(f \sm e) \se g \cup -(f \sm e)$, hence $-g \cap (f \sm e) \se p_1$.

 Now write \begin{align*} b &  = (b\cap (e\cap f))\cup (b\cap (-(e\cap f))) \\ 
& = ((g\cap e \cup (-g)\cap f) \cap (e\cap f)) \cup ((g\cap e \cup (-g)\cap f) \cap (-(e\cap f)))\\
& = g\cap (e\cap f) \cup (-g) \cap (e\cap f) \cup (g\cap (e\sm f) \cup (-g)\cap (f\sm e))\\ 
& = (e\cap f ) \cup ((p_0\cap (e\sm f)) \cup (p_1\cap (f\sm e))).
\end{align*} The last line is a member of $A$. \end{proof}

We now use a version of Proposition 3.6  from \cite{kos}.  We include a detailed proof here for completeness. 

\begin{lem}  Let $\X$ be maximal ideal-independent in $A$. Suppose that $A, \X,$ and $P(A, u)$ are all subsets of a c.t.m. $M$ of ZFC.  Suppose that $G$ is $P(A, u)$-generic over $M$, and let $g = \bigcup \set{ p_0 : (p_0, p_1)\in G}$.  Then $\X$ is still maximal ideal independent in $\ang{A\cup\set{ g}}$ (as viewed in $M[G]$). 
\end{lem}
\begin{proof}  Let $e, f\in A$.  For any $p\in P(A, u)$ we define $p^* = (e\cap p_0) \cup (f\cap p_1)$, $a_p = \kappa \sm (p_0\cup p_1)$,  $e_p =a_p \cap e,$ and $f_p = a_p\cap f$.  We say that an element $\alpha\in \kappa$ is \emph{permitted by }$p$ iff there is a $q\leq p$ with $\alpha\in q^*$. We define a subset $D_{e,f}$ of $P(A, u)$ in $M$ as follows. 

$p\in D_{e,f}$ iff one of the following conditions holds: \begin{enumerate}
\item $p_0\cup p_1 \supseteq e\sdiff f $
\item  $\exists n\in \w \,\exists x_0, \ldots, x_n\in \X\, \left[ x_0\se p^*\cup x_1\cup \ldots \cup x_n\right]$
\item $\exists n\in \w \,\exists x_0, \ldots, x_n\in \X\, \left[ p^*\cup  a_p\se x_0\cup \ldots \cup x_n\right]$

%
%
\end{enumerate}

We claim \begin{equation*}\tag{$*$} \text{If } p\in P(A, u)\text{ and }x\not\in u, \text{ then there is a } q\leq p \text{ such that } x\se q_0\cup q_1
\end{equation*} In fact, let $q_0 = p_0\cup (-p_1\cap x)$ and let $q_1 = p_1\cup (-q_0\cap x)$. 

First we claim that $D_{e,f}$ is dense in $P(A, u)$.  So let $p\in P(A, u)$. One of the following holds 
\begin{enumerate}[(i)]  \item $e_p\cap f_p\in u$
\item$ \kappa\sm(e_p\cup f_p) \in u$
\item $e_p\sm f_p\in u$
\item $f_p\sm e_p\in u$. 
\end{enumerate}

Note that $e_p\sm f_p = a_p \cap (e\sm f)$, $f_p\sm e_p = a_p \cap(f\sm e), $ and $e_p\sdiff f_p = a_p\cap (e\sdiff f)$.  If (i) or (ii) is the case, then $e_p\sdiff f_p\not\in u$, so also $e\sdiff f\not \in u$ (as $p_0\cup p_1\not\in u$). By $(*)$ there is a $q\leq p $ such that $q_0\cup q_1 \supseteq e\sdiff f$, so that (1) of the definition of $D_{e,f}$ is satisfied. 

Next, suppose that (iii) is the case.  Then also $e\sm f\in u$; by $(*)$  there is a $q\leq p$ such that $-(e\sm f)\se q_0\cup q_1$,  so that $a_q\se e\sm f$.  Now by maximality of $\X$ in $A$ we have that for some $n\in \w$ and some  $x_0, \ldots, x_n\in X$,    \be[(i)]\setcounter{enumi}{4}
\item $x_0\se q^*\cup x_1\cup \ldots \cup x_{n}$, or 
\item $q^* \se x_0\cup \ldots \cup x_n.$
\ee
If (v)  is the case, then condition (2)  in the definition of $D_{e,f}$ is satisfied.  So suppose that (vi) is  the case.  Again, by maximality of $\X$ in $A$, there is an $m\in \w$ and some $y_0, \ldots, y_m\in \X$ such that either :\be[(i)]\setcounter{enumi}{6}\item $a_q \se y_0\cup \cdots \cup y_m$, or 
\item $y_0\se y_1\cup \ldots \cup y_m\cup a_q$.
\ee If (vii) holds then  $q^*\cup a_q \se x_0\cup \ldots \cup x_n \cup y_0\cup \ldots\cup y_m$, so condition (3) of the definition of $D_{e,f}$ is satisfied.  Suppose then that (viii) holds. 
\bi \item  Case 1. $a_q\cap y_0\in u$. Then $a_q\sm y_0\not\in u$.  Let $r_0 = q_0$ and $r_1  = q_1\cup (a_q\sm y_0)$.      We claim that $r^*\cup a_r\se y_0\cup x_0\cup \ldots\cup x_n$, so  $r$ satisfies (3) in the definition of $ D_{e,f}$. In fact, $a_r = a_q\cap y_0\se y_0.$ Now recall $r^* = (e\cap r_0)\cup (f\cap r_1)$.    Note that   $r_0\sm q_0 = \0 $ and $r_1\sm q_1\se a_q.$  In particular, since $a_q\se e\sm f$, $f\cap r_1 = f\cap q_1$. Hence $r^* = q^*$, and by (vi) $q^*\se x_0\cup \ldots\cup x_n$. So $r$ satisfies condition (3) of $D_{e,f}$.


\item Case 2. $a_q\cap y_0\not\in u$. Then let $r_0 = q_0 \cup (a_q\cap y_0)$ and let $r_1 = q_1$. Now using (viii) we have that $y_0\se y_1 \cup \ldots \cup y_m \cup (a_q\cap y_0)$.  Also $a_q\cap y_0\se a_q \se e$, so $a_q\cap y_0\se r^*$. Thus we have $y_0\se y_1\cup \ldots \cup y_m \cup r^*$.  So condition (2) in the definition of $D_{e,f}$. is satisfied. 
\ei

The case when $f_p\sm e_p\in u$ is treated similarly.  Thus we have proved that the sets $D_{e,f}$ are indeed dense.

Now suppose that $b\in \ang{A\cup \set{ g}}$, we will show that $\X\cup\set{ b} $ is not ideal independent.  Write $b = (e\cap g) \cup (f\cap (-g))$ for some $e, f\in A$.  Now let $p\in D_{e,f}$ be such  that $p\in G$.  Note that $p_0\se g$. Also  $p_1\se (-g)$, for, suppose that $q\in G$; we want to show that $p_1\cap q_0 = 0$. Choose $r\in G$ such that $r\leq p,q$. Then $p_1\cap q_0\se r_1\cap r_0 = 0$.  So $p^*\se b$.  We consider cases according to the definition of $D_{e,f}$. 
\bi\item Case 1. $p_0\cup p_1\supseteq e\sdiff f$.  Then \lemref{ina} gives that $b\in A$, so $\X\cup \set{ b}$ is not independent by maximality of $\X$ in $A$.

\item Case 2.   $\exists n\in \w \,\exists x_0, \ldots, x_n\in \X\, \left[ x_0\se p^*\cup x_1\cup \ldots \cup x_n\right]$. Then $x_0\se b \cup x_1\cup \ldots \cup x_n$. 

\item Case 3. $\exists n\in \w \,\exists x_0, \ldots, x_n\in \X\, \left[ p^* \cup a_p\se x_0\cup \ldots \cup x_n\right]$. Cleary $b\cap (p_0\cup p_1) = p^*,$ so $b\se p^*\cup a_p$. So also $b\se x_0\cup \ldots \cup x_n$. 
\qedhere \ei\end{proof}

\begin{prop}[CH] Assuming the continuum hypothesis, there is an atomless BA $B$ such that $\smm(B) = \w$ and $\uu(B) = \w_1$.    
\end{prop} 
\begin{proof}  First we make some definitions.   Suppose that $A$ is a countable atomless subalgebra of $\pset{\w}$ and $u$ is an ultrafilter of $A$.  All ultrafilters are assumed to be nonprincipal.  Let $P(A, u)$ be defined as above. For each $a\not\in u$, define $$D_a = \Big\{ (p_0, p_1)\in P(A, u): a \se (p_0\cup p_1), p_0\sm a\not= \0 \not= p_1\sm a \Big\}. $$  We claim that each $D_a$ is dense in $P(A, u)$.  If $p = (p_0, p_1)\in P(A, u)$, then we have that $b: = p_0\cup p_1\cup a\not\in u$.  Because $A$ is atomless, there are disjoint  $x_0, x_1\se \w\sm b$ such that each $x_i\not\in u$.  Define $q_0 = p_0\cup x_0$ and $q_1 = p_1\cup x_1\cup (a\sm p_0)$. We have $q_0\sm a \not = 0$ since $x_0\se \w\sm a$, hence $x_0 = x_0\sm a \se q_0\sm a$.  Similarly $q_1\sm a \neq 0$.  So $(q_0, q_1)$ is an extension of $p$ in $D_a$.  

Next, for $i\in \w$ we define $$E_i = \set{ (p_0, p_1)\in P(A, u): i\in p_0\cup p_1}.$$ Since $u$ is nonprincipal, $\set{ i}$ is not a member of $u$ for any $i\in \w$.  Thus if $p = (p_0, p_1)\not\in E_i$ then $(p_0\cup\set{i}, p_1)$ is an extension of $p$ that is a member of $E_i$.

For $e,f\in A$ let $D_{e,f}$  be as in the above lemma.  Let  $$\D = \set{ D_a : a\not\in u}\cup \set{ D_{e,f}: e,f\in A}\cup\set{ E_i: i\in\w}.$$ Since $\D$ is a countable collection of dense sets, there is a filter $G\se P(A, u)$ such that $G$ has nonempty intersection with every member of $\D$. Given such a $G$ we will say that $x$ is \emph{generic for} $P(A, u)$  iff  $x = \bigcup_{(p_0, p_1)\in G} p_0$. By the proof of the above lemma, such a generic $x$ does not extend any maximal ideal-independent set of $A$.  We shall show, however, that  $x\not\in A,$ and $u$ does not generate an ultrafilter in $\ang{A\cup\set{x}}$.  First, suppose for a contradiction that  $x\in A$. Then either $x\in u$ or $-x\in u$.  If $-x\in u$ then $D_x\in \D$, so choose $p= (p_0, p_1)\in D_x\cap G$. By definition of $x$ we have $p_0\se x$.  But $p\in D_x$, so also $p_0\sm x\not=\0$, a contradiction.  We reach a contradiction similarly if $x\in u$. In fact, the same argument works
since if $p\in D_{-x} \cap G$ then $p_1\se -x$.  For, if $q\in G$, choose $r\in G$ with $r\leq p, q$. Then
$q_0\cap p_1\se r_o\cap r_1 = \0$.  So $p_1\cap q_0 = \0$. Hence $p_1\cap x = \0$.

 Next, suppose that $u$ were to generate an ultrafilter in $\ang{A\cup\set{x}}$.  So there is an $a\in A\sm u$ such that either $x\leq a$ or $-x\leq a$.  If $x\leq a$ then consider $(p_0, p_1)\in G\cap D_a$.  We claim that $x = x\cap a = p_0\cap a \in A$, a contradiction.  In fact, clearly $x\cap a \supseteq p_0\cap a$.  For the other inclusion, consider an arbitrary $q \in G$ and let $r\in G$ be such that $r\leq q, p$. Then since $p\in D_a$ we get $q_0\cap a \se r_0\cap (p_0\cup p_1)\cap a\se p_0$, since $r_0\cap r_1 = 0$ and $p_1\se r_1$.  Thus $x\cap a\se p_0\cap a$.    To carry out a symmetrical argument in case  $-x\leq a$ we just need to see that $-x = \bigcup_{(p_0,p_1)\in G} p_1$.   For $(\se)$, suppose that $i\in -x$.  Let $p\in G\cap E_i$.  So $i\in p_0\cup p_1$.   We must have $i\not\in p_0$ or else $i\in x$, so $i\in p_1$.  For the opposite inclusion, suppose that $p\in G$ and $i\in p_1$.  Letting $q\in G$ be arbitrary it suffices to show that $i\not\in q_0$. Find $r\in G$ such that $r\leq p,q$.  Then $r_0\cap r_1 = \0$ implies that $r_0\cap p_1 = \0$, so $i\not \in r_0$.  Now, because  $r_0\supseteq q_0$, we see that also $i\not\in q_0$.

As a final preliminary, we would like to see that $\ang{A\cup \set{ x}}$ is atomless (since $A$ is). Suppose for a contradiction that $x\cdot a$ is an atom for some $a\in A$.  If $a\not\in u$ then $x\cdot a  = p_0\cdot a $ for $(p_0, p_1)\in D_{a}\cap G$ (as proved and used above). As $p_0\cdot a\in A$ this contradicts the fact that $A$ is atomless.  So $a\in u$.    Now, consider $p: = (p_0, p_1)\in D_{-a}\cap G$. We have that $p_0\sm (-a) = p_0\cdot a$ is not $\0$.  Also $p_0\cdot a \not\in u$.  So there is a $q\in D_{a\cdot p_0}\cap G$.  Then as above we have $q_0\cdot (a\cdot p_0) = x\cdot (a\cdot p_0)$.   Note that $x\cdot p_0 = p_0$, so the set on the right hand side is equal to $p_0\cdot a$, hence is nonempty, and is in fact equal to the atom $x\cdot a$. But the set on the left hand side is in $A$, a contradiction.  If the $-x\cdot a $ were assumed to be the atom, a symmetric argument yields a contradiction.

Let $\ang{\ell_\alpha: \alpha<\w_1}$ enumerate the limit ordinals below $\w_1$. Partition $\w_1$ into the sets $\set{ M_i : i\in \w_1}$, with each part of size $\w_1$. For each $i\in \w_1$ let $\ang{k^i_\alpha : \alpha<\w_1}$ enumerate $M_i\sm (\ell_i + 1)$.  Now we construct a sequence $\ang{A_\alpha: \alpha<\w_1}$ of countable atomless subalgebras of $\pset{\w}$ as follows. Let $A_0$ be an arbitrary denumerable atomless subalgebra of $\pset{\w}$.   For any limit ordinal $\alpha = \ell_i$ let $A_\alpha = \bigcup_{\beta<\alpha} A_\beta$ and let $\ang{u^i_\beta : \beta<\w_1}$ enumerate all the nonprincipal ultrafilters on $A_\alpha$.  Now suppose $\alpha$ is the successor ordinal $\gamma + 1$. If  $\gamma = k^i_{\beta},$ we proceed as follows.  Note that $\ell_i<k^i_\beta$ and so $u_\beta^i\se A_\gamma$. Let  $\overline{u^i_\beta}$ denote the filter on $A_{\gamma}$ generated by $u^i_\beta$.    If $\overline{u^i_\beta}$ is not an ultrafilter or if $\gamma$ is not in any of the sets $M_i\sm (\ell_i +1)$ let $A_\alpha = A_\gamma$.  If $\overline{u^i_\beta}$ is an ultrafilter then we let  $x_\gamma$ be generic for $P\paren{A_\gamma, \overline{u^i_\beta}}$.  Define $A_\alpha = \ang{A_\gamma \cup\set{ x_\gamma}}$.  Note that $\overline{u_\beta^i}$ does not generate an ultrafilter on $A_\alpha$.

Now define $B= \bigcup_{\alpha<\w_1} A_\alpha$.  $B$ is atomless as it is a union of atomless algebras.   Suppose that some countable $X\se B$ generates an ultrafilter on $B$. Then pick a limit ordinal $\alpha = \ell_i< \w_1$ such that $X\se A_\alpha$.   So $X$ generates an ultrafilter of $A_\alpha$; say it generates $u^i_\beta$. Let $\gamma = k^i_\beta$.  Then by construction, $X$ does not generate an ultrafilter on $A_{\gamma + 1}$, contradiction.  Therefore $\abs{B} = \w_1 = \uu(B)$.  

Next we claim  that  $\smm(B) =\w$.  We shall show that if $Y\se A_0$ is any maximal ideal-independent set, then $Y$ is still maximal ideal-independent in $B$. Since any member of $B$ is a member of $A_\alpha$ for some $\alpha< \w_1$, it suffices to show that $Y$ is still maximal independent in each $A_\alpha$.  Suppose to the contrary, and let $\alpha$ be minimal such that there is some $x\in A_\alpha\sm Y$ such that $Y\cup \set{ x}$ is still ideal-independent.  If $\alpha$ is a limit ordinal then $x\in A_\beta$ for some $\beta<\alpha$.  By minimality of $\alpha$, we see that $Y\cup \set{ x}$ is not ideal-independent, a contradiction.  If $\alpha = \beta + 1$, then the above lemma implies that, since $Y$ is maximal ideal independent in $A_\beta$, it remains so in $A_{\alpha}$, a contradiction.  Thus $\smm(B) = \w$ as claimed.  \end{proof}

\begin{lem} Suppose that $A$ is a subalgebra of $\pset{\kappa}$ and suppose that $u$ is an ultrafilter of $A$ (nonprincipal as always).  Suppose that $C= \ang{c_i: i<\xi } \se A$ is a maximal free sequence of $A$ such that $c_{i}\se c_j$ for each $i>j\in \xi$.  There is a family, $\E,$  of dense subsets of $P(A, u)$ such $C$ remains maximal in $\ang{A\cup\set{g}}$ whenever  $g = \bigcup_{(p_0,p_1)\in G} p_0$ for a filter $G\se P(A, u)$ intersecting each member of $\E$. Moreover $\abs{\E}\leq \abs{A}$.   \end{lem}
\begin{proof}  We will use the following fact several times: \begin{equation*}\tag{$*$}
\forall a\in A \, \Big(\exists i\in \xi\, \big[a\se (\kappa\sm c_i)\big] \text{ or }\exists i<j\in \xi \, \big[(c_i\sm c_j)\se a \big] \text{ or } \kappa\sm c_0\se a \Big)
\end{equation*}
To see this, suppose that $a\in A$.  Clearly the desired conclusion holds if $a = \0$ or $a= \kappa$; so suppose that $a\neq \0, \kappa$.  By maximality of $C$ we have that either

\be[A.] \item $\exists F\in \fsub\xi$ such that $\paren{\bigcap_{i\in F} c_i} \cap a  = \0$, or 
\item $ \exists F, G\in \fsub\xi$, with $\forall i\in F\forall j\in G[i<j]$ such that $\paren{\bigcap_{i\in F} c_i }\cap \paren{\bigcap_{j\in G} \kappa\sm c_j}\cap (\kappa\sm a) = \0 $
\ee

If $A$ holds then $F\neq  \0$ since $a\neq\0$ and then $c_{\max F}\cap a = \0$ so that $a\se \paren{\kappa\sm c_{\max{F}}}$, hence the first part of ($*)$ holds. 

If $B$ holds then $F\neq \0 $ or $G\neq \0$ since $a\neq \kappa$. If $F\neq \0\neq G$ then $(c_{\max F}\sm c_{\min G})\se a$, giving the second condition of $(*)$. If $F\neq \0 = G$ then $c_{\max{F}}\se a$, giving the second condition of $(*)$ again. Finally if $F = \0 \neq G$ then $(\kappa\sm c_{\min G}\se a$, giving the second or third condition of $(*)$. 

%

 Let $e, f\in A$. For any $p\in P(A, u)$ we define $p^* = (p_0\cap e)\cup (p_1\cap f)$, and $a_p= \w\sm (p_0\cup p_1)$.  We define a subset $E_{e,f}$ of $P(A, u)$ as follows. 

$p\in E_{e,f}$ iff one of the following conditions holds: \be
\item $p_0\cup p_1\supseteq e\sdiff f$
\item $\exists i<j\in \xi\, [p^*\supseteq c_i\sm c_j]$
\item $\exists i\in \xi \, [p^*\cup a_p \se \kappa \sm c_i]$ 
\item $\kappa\sm c_0\se p^*$. 
\ee
We claim that $E_{e,f}$ is dense.  Let $p\in P(A, u)$.  One of the following holds 
\begin{enumerate}[(i)]  \item $e_p\cap f_p\in u$
\item$ \w \sm(e_p\cup f_p) \in u$
\item $e_p\sm f_p\in u$
\item $f_p\sm e_p\in u$. 
\end{enumerate}

If (i) or (ii) is the case, then $e_p\sdiff f_p\not\in u$, so also $e\sdiff f\not \in u$ (as $p_0\cup p_1\not\in u$). Thus we can extend $p$ to a condition $q$ such that $q_0\cup q_1 \supseteq e\sdiff f$, so that (1) of the definition of $E_{e,f}$ is satisfied.

Next, suppose that (iii) is the case.  Then also $e\sm f\in u$, so we can first extend $p$ to some condition $q$ so that $a_q\se e\sm f$.  Now $q^*\in A$, so, by $(*)$, either \be[(i)]\setcounter{enumi}{4}
\item $\exists i<\xi \, [q^*\se \kappa\sm c_i ],$ or 
\item $\exists i<j\in\xi\, [q^*\supseteq c_i\sm c_j]$, or 
\item $\kappa\sm c_0\se q^*$.
\ee
If (vi) holds then $q$ is in $E_{e,f}$ by virtue of condition (2).  If (vii), then $q$ is in $E_{e,f}$ by virtue of (4).   So we assume now that (v) is the case, and fix $i\in \xi$ as guaranteed by (v). Now also $a_q\in A$, so either 
\be[(i)]\setcounter{enumi}{7}
\item $\exists j<\xi \, [a_q\se \kappa\sm c_j ],$ or 
\item $\exists j<k\in\xi\, [a_q\supseteq c_j\sm c_k]$, or 
\item $\kappa\sm c_0\se a_q$. 
\ee  First suppose that (viii) holds. Then $a_q\cup q^*\se (\kappa\sm c_i)\cup (\kappa \sm c_j) = \kappa \sm (c_i \cap c_j) = \kappa \sm c_{\max\set{ i,j}}$,  so $q\in E_{e,f}$ by virtue of condition (3).  Next assume that (ix) holds and fix $j<k\in \xi$ as in that case.   We consider two cases. 
\bi \item Case 1. $(c_j\sm c_k)\in u$.  Then extend $q$ to a condition $r$ such that $r_0= q_0$, and $r_1 = q_1\cup (-q_0\cap -(c_j\sm c_k)). $ Then $-(c_j\sm c_k) \se r_0\cup r_1$, so $a_r\se c_j\sm c_k$. Note that $r_1\sm q_1 \se a_q\se e\sm f$, so $(r_1\sm q_1)\cap f = 0$. Then $r^* = (r_0\cap e) \cup (r_1\cap f)  = (q_0\cap e) \cup (r_1\cap f)$, and $(r_1\sm q_1) \cap f = \0$, so in fact $r^* = q^*$.  Recall that $q^*\se (\kappa \sm c_i)$ so $r^*\cup a_r\se (\kappa \sm c_{\max\set{i,k}})$.  Thus condition (3) holds for $r$.

\item Case 2. $(c_j\sm c_k)\not\in u$. Then we extend $q$ to a condition $r$ so that $r_0 = q_0 \cup (c_j \sm c_k)$ and $r_1 = q_1$.  Recall that $(c_j\sm c_k)\se a_q\se e$, so $r^*\supseteq (r_0\cap e)\supseteq (c_j\sm c_k)\cap e = c_j\sm c_k$.  Thus $r$ satisfies condition (2) in the definition of $E_{e,f}$.

\ei

Finally suppose that (x) is the case.  Again we consider two cases. 

\bi
\item Case 1.  $a_q \cap c_0 \not\in u$.  Then we extend $q$ to a condition $r$ where $r_0 = q_0$ and $r_1  = q_1\cup (a_q\cap c_0)$.  Then $a_r\se (\kappa \sm c_0)$.  Also $r^* = q^*$ by the same argument as in Case 1 above.  So $a_r\cup r^*\se (\kappa\sm c_i)$, and $r$ satisfies condition (3) of the definition of $E_{e,f}$.

\item Case 2. $a_q \cap c_0 \in u$.  Then we extend $q$ to a condition $r$ by setting $r_0 = q_0\cup (a_q\sm c_0)$ and $r_1 = q_1$.   Then $r^* \supseteq r_0\cap e \supseteq \kappa\sm c_0 $, so condition (4) in the definition of $E_{e,f}$ holds.

\ei

Thus the sets $E_{e,f}$ are dense.  Let $\E = \set{ E_{e,f}: e,f\in A}$. Clearly $\abs{\E}\leq \abs{A}$.  Suppose that $G$ is a filter that intersects each member of $\E$, and let $g = \bigcup_{(p_0,p_1)\in G} p_0$.  We must show that $C$ is still maximal in $\ang{A\cup \set{ g}}$.  Letting $b\in \ang{A\cup \set{g}}$ we can write $b = (g\cap e) \cup( f\sm g)$ for some $e,f\in A$.  Let $p\in G\cap E_{e,f}$;  we will show that $C^{\frown}\set{ b} $ is no longer free, considering cases according to the definition of $E_{e,f}$. \bi
\item Case 1. $p_0\cup p_1 \supseteq e\sdiff f .$ By \lemref{ina}, in this case $b\in A$. So $b$ does not extend $C$  by maximality in $A$.

\item Case 2.  $\exists i<j\in \xi\, [p^*\supseteq c_i\sm c_j]$. We have that $p^*\se b$, so also $c_i\sm c_j\se b$. Then $(c_i)\cap (\kappa \sm c_j) \cap (\kappa \sm b) = \0$, so $b$ does not extend $C$.

\item Case 3. $\exists i\in \xi \, [p^*\cup a_p \se \kappa \sm c_i]$.   
Clearly $b\cap (p_0\cup p_1) = p^*,$ so $b\se p^*\cup a_p$.  So $b\se \kappa \sm c_i$.  Thus $c_i\cap b = \0$, and again $b$ does not extend $C$.

\item Case 4.  $\kappa\sm c_0\se p^*$. Since $p^*\se b$, also  $\kappa \sm c_0 \se b$ so $(\kappa\sm c_0)\cap (\kappa\sm b) = \0$.\qedhere
\ei 
\end{proof}

\begin{prop}[CH]Assuming CH there is an atomless Boolean algebra $A$ such that $\ff(A) = \w$ and $\uu(A) = \w_1$. 
\end{prop}\begin{proof} As above, we construct a sequence of countable, atomless subalgebra of $\pset{\w}$, $\ang{A_\alpha: \alpha<\w_1}$, such that if $B: = \bigcup_{\alpha<\w_1} A_\alpha$, $\uu(B) = \w_1$.  The differences are as follows.  \be \item In $A_0$ there is a countable maximal free sequence $C = \ang{c_i: i\in\w}$ such that $c_{j}\se c_i$ for each $i<j\in \w$.  
\item The sets $E_{e,f}$ of the previous lemma take the place of the $D_{e,f}$ in the definition of $\D$. 
\ee
We claim that $\ff(B) = \w$.  Because $B$ is atomless, it suffices to show that $C$ is still maximal in $B$. Since each member of $B$ is also a member of some $A_\alpha$ for $\alpha<\w_1$, it suffices to show that $C$ is maximal in each $A_\alpha$.  Suppose not and let $\alpha<\w_1$ be minimal such that for some $x\in A_\alpha$,  $C^\frown\set{ x}$ is still a free sequence.   Clearly $\alpha$ is a successor ordinal, say $\alpha = \beta + 1$. By minimality of $\alpha$, $C$ is maximal in $A_\beta$.  The above lemma guarantees that $C$ is still maximal in $A_\alpha$, contradiction. \end{proof}

\begin{cor}[CH] Assuming CH there is an atomless Boolean algebra $A$ such that $\smm(A) = \ff(A) = \w<\w_1 = \uu(A)$. 
\end{cor}

\newpage

\bibliographystyle{amsalpha} 
\bibliography{/Users/eduardo/Documents/TeX/atomless.bib}

\end{document}